\documentclass[12pt]{amsart}
\usepackage{amsthm,amsmath,a4}

\newtheorem{thm}{Theorem}
\newtheorem{cor}[thm]{Corollary}
\newtheorem{lem}[thm]{Lemma}
\newtheorem{prop}[thm]{Proposition}
\theoremstyle{definition}

\theoremstyle{remark}
\newtheorem{rem}[thm]{Remark}
\numberwithin{equation}{section}
\newcommand{\PP}{\mathcal{P}}
\newcommand{\QQ}{\mathcal{Q}}
\newcommand{\Ass}{\textnormal{Asso}}

\newcommand{\Sub}{\textnormal{Sub}}
\newcommand{\Quo}{\textnormal{Quo}}
\newcommand{\rk}{\textnormal{CB-rk}}

\begin{document}

\address{IRMAR, Campus de Beaulieu, 35042 Rennes CEDEX, France}
\email{yves.decornulier@univ-rennes1.fr}
\subjclass[2000]{Primary 13C05; Secondary 13E05}

\title{The space of finitely generated rings}
\author{Yves Cornulier}%
\date{\today}


\begin{abstract}The space of marked commutative rings on $n$ given generators is a compact metrizable space. We compute the Cantor-Bendixson rank of any member of this space. For instance, the Cantor-Bendixson rank of the free commutative ring on $n$ generators is $\omega^n$, where $\omega$ is the smallest infinite ordinal. More generally, we work in the space of finitely generated modules over a given commutative ring.
\end{abstract}
\maketitle
\section{Introduction}

All rings in the paper are commutative with unity. The space $\mathcal{R}_n$ of finitely generated rings marked on $n$ generators is by definition the set of pairs $(A,(x_1,\dots,x_n))$ where $A$ is a ring endowed with a family of ring generators $(x_1,\dots,x_n)$, up to marked ring isomorphism. This is a topological space, where a prebasis of neighbourhoods of the marked ring $(A,(x_1,\dots,x_n))$ is given by the sets $V_P=V_P(A,(x_1,\dots,x_n))$, for $P\in\mathbf{Z}[x_1,\dots,x_n]$, defined as follows: if $P=0$ (respectively $P\neq 0$) in $A$, $V_P$ is the set of marked rings on $n$ generators in which $P=0$ (resp. $P\neq 0$), where $P$ is evaluated on the $n$-tuple of marked generators. The space $\mathcal{R}_n$ is a compact metrizable, totally disconnected topological space, which is moreover countable by noetherianity. Hence every element of this space has a well-defined Cantor-Bendixson rank. Informally, this is the necessary (ordinal) number of times we have to remove isolated points so that the point itself becomes isolated. It is easy to check that the Cantor-Bendixson rank of a given marked ring does not depend on the choice of generators, and only on the isomorphism class within rings. See the beginning of Section \ref{sec:cb} for more details.

To every noetherian ring, we can associate an ordinal-valued length, characterized by the formula
$$\ell(A)=\sup\{\ell(B)+1|B\text{ proper quotient of }A\}\quad\text{(Agreeing }\ell(\{0\})=0).$$
This length was introduced in \cite{Bass,Gull} and further studied in \cite{Krau,Bro1,Bro2}. It can be computed in a more explicit way (see Section \ref{lengthbis}). In particular, if the Krull dimension of $A$ is an ordinal $\alpha$, then $\omega^\alpha\le\ell(A)<\omega^{\alpha+1}$, where the left-hand inequality is an equality if and only if $A$ is a domain.

This length provides an obvious upper bound for the Cantor-Bendixson rank, but this is not optimal as every finite nonzero ring is isolated and therefore has zero Cantor-Bendixson rank, while it has finite but nonzero length. Accordingly we introduce the reduced length, an ordinal-valued function characterized by the formula
$$\ell'(A)=\sup\{\ell'(B)+1|B\text{ quotient of }A\text{ with non-artinian kernel}\}.$$
In the context of finitely generated rings, non-artinian just means infinite. A precise formula is given in Section \ref{reducedo}. In particular, if the Krull dimension of $A$ is a finite number $d$, then $\omega^{d-1}\le\ell'(A)<\omega^{d}$, where the left-hand inequality is an equality if $A$ is a domain, and agreeing $\omega^{-1}=0$.

\begin{thm}
Let $A$ be a finitely generated ring. Then its Cantor-Bendixson rank
coincides with its reduced length.\label{main}
\end{thm}

\begin{cor}
If $A$ is a finitely generated domain of Krull dimension $d$, then its Cantor-Bendixson rank is $\omega^{d-1}$.\label{cor:cbrkdom} 
\end{cor}

Actually, every ring can be viewed as a module over itself generated by one
element, and ideals and submodules coincide; therefore it is a more general point of view if we consider the space of modules generated by $k$ marked elements over $\mathbf{Z}[x_1,\dots,x_n]$ (or even any commutative ring), the case $k=1$ corresponding to the space of rings marked by $n$ elements. Actually the point of view and language of modules (extensions, etc.) are very natural and useful in this context and it would have been awkward to restrict to rings. The definition of length and reduced length is extended to modules in the next sections, and in Section \ref{sec:cb} we will obtain Theorem \ref{main} as a particular case of Theorem \ref{cb}, which holds in the general context of modules over finitely generated rings. Comments on the case of modules over some infinitely generated rings are included in Section \ref{sec:other}.

\begin{rem}
It also makes sense to talk about the set of {\it integral domains} generated by $n$ given elements; this can be viewed as a closed subset of $\mathcal{R}_n$ and corresponds bijectively to $\text{Spec}(\mathbf{Z}[X_1,\dots,X_n])$. It is proved in \cite{CDN} that, {\it inside this space}, the Cantor-Bendixson rank of any $d$-dimensional domain is $d$ and in particular is finite (this contrasts with Corollary \ref{cor:cbrkdom}).
\end{rem}


\section{Length}\label{lengthbis}

Let $A$ be a ring (always assumed commutative). If $M$ is an $A$-module, define $\Lambda(M)$ as the set of elements $m\in M$ such that $Am$ has finite length. This is a submodule of $M$; if moreover $M$ is noetherian, then $\Lambda(M)$ itself has finite length.

Let $\PP$ be a prime ideal in $A$ and $M$ a noetherian $A$-module. As usual, $A_\PP$ denotes the local ring of $A$ at $\PP$ ($A_\PP=S^{-1}A$, where $S=A-\PP$), and $\ell_{A_\PP}$ denotes the length function for finite length $A_\PP$-modules.

Define 
$$\ell_\PP(M)=\ell_{A_\PP}(\Lambda(M\otimes_AA_\PP)).$$

Recall that $\Ass_A(M)$ is defined as the set of prime ideals $\PP$ of $A$ such that $A/\PP$ embeds as a submodule of $M$. It is known to be finite if $M$ is noetherian, and non-empty if moreover $M\neq 0$.

\begin{lem}
We have $\ell_\PP(M)>0$ if and only if $\PP\in\Ass_A M$.\label{lengthass}
\end{lem}
\begin{proof}
If $\PP\in\Ass_A M$, then $A/\PP$ embeds into $M$; by flatness of $A_\PP$, this implies that $A_\PP/\PP A_\PP$ embeds into $M\otimes A_\PP$; so that $\Lambda(M\otimes A_\PP)$ is non-zero and therefore has non-zero length.

Conversely suppose that $\Lambda(M\otimes A_\PP)$ is non-zero. Then there exists an associated ideal $\QQ'\in\Ass_{A_\PP}(\Lambda(M\otimes A_\PP))$. We can write $\QQ'$ as $\QQ A_\PP$, for some prime ideal $\QQ$ of $A$ contained in $\PP$. Now $A_\PP/\QQ A_\PP$ embeds into $\Lambda(M\otimes A_\PP)$ and therefore has finite length; this forces $\QQ=\PP$. In other words, we have just proved that $$\Ass_{A_\PP}(\Lambda(M\otimes A_\PP))=\{\PP A_\PP\}.$$ On the other hand, $$\Ass_{A_\PP}(\Lambda(M\otimes A_\PP))\subset \Ass_{A_\PP}(M\otimes A_\PP)$$ $$=\{\QQ A_\PP|\QQ\in\Ass_A(M),\QQ\subset\PP\}.$$ Therefore $\PP\in\Ass_A(M)$. 
\end{proof}

\begin{lem}
Consider a short exact sequence $0\to K\to M\to N\to 0$ of noetherian $A$-modules. Let $\PP$ be a prime ideal in $A$. Suppose that the $A_\PP$-module $K\otimes A_\PP$ has finite length. Then $\ell_\PP(M)=\ell_\PP(N)+\ell_\PP(K)$.\label{lengthext}
\end{lem}
\begin{proof}
By flatness of $A_\PP$, we get an exact sequence
$$0\to K\otimes A_\PP\to M\otimes A_\PP\to N\otimes A_\PP\to 0,$$
which by the assumption on $K$ induces an exact sequence
$$0\to K\otimes A_\PP\to \Lambda(M\otimes A_\PP)\to \Lambda(N\otimes A_\PP)\to 0.$$ 
\end{proof}

\begin{lem}
Let $\PP$ be a prime ideal in a noetherian ring $A$. Equivalences:
\begin{itemize}
\item[(i)] The sequence $(\ell_\PP(A/\PP^n))$ is bounded; 
\item[(ii)] $\PP$ is a minimal prime ideal.
\end{itemize}\label{lengthinfinity}
\end{lem}
\begin{proof}
$\ell_\PP(A/\PP^n)=\ell_{A_\PP}(\Lambda(A_\PP/\PP^n A_\PP))=\ell_{A_\PP}(A_\PP/\PP^n A_\PP)$.

If $\PP$ is minimal, then $\PP A_\PP$ is the radical of the artinian ring $A_\PP$, so that $\PP^n A_\PP$ is eventually zero and therefore the sequence $A_\PP/\PP^n A_\PP$ eventually stabilizes.

Conversely, if the sequence above is bounded, then it is stationary, so that in the local ring $A_\PP$, $(\PP A_\PP)^n=(\PP A_\PP)^{n+1}$ for some $n$. By Nakayama's Lemma, this forces $(\PP A_\PP)^n=0$, so that $A_\PP$ is actually artinian, i.e. $\PP$ is a minimal prime ideal.
\end{proof}

Let now $A$ be a ring. Say that a prime ideal $\PP$ in $A$ is conoetherian if $A/\PP$ is noetherian. Define, for every conoetherian prime ideal $\PP$, its coheight as the ordinal
$$\text{coht}(\PP)=\sup\{\text{coht}(\QQ)+1|\QQ\text{ prime ideal properly containing }\PP\}.$$
The noetherianity assumption makes this definition valid.

If $M$ is a noetherian $A$-module, then every $\PP\in\Ass_A(M)$ is conoetherian. Define, for every ordinal $\alpha$,
$$\ell_\alpha(M)=\sum\ell_\PP(M),$$
where $\PP$ ranges over conoetherian prime ideals of coheight $\alpha$ in $A$.
This is a finite sum as $\ell_\PP(M)\neq 0$ only when $\PP\in\Ass_A(M)$. Besides, the (ordinal-valued) Krull dimension of $M$ is defined as $\sup\text{coht}(\PP)$, where $\PP$ ranges over prime ideals of $A$ containing the annihilator of $M$ (or, equivalently, over all associated primes of $M$).

\begin{lem}
Consider a short exact sequence $0\to K\to M\to N\to 0$ of noetherian $A$-modules. Let $\alpha$ be an ordinal. Suppose that $K$ has Krull dimension $\le\alpha$. Then $\ell_\alpha(M)=\ell_\alpha(N)+\ell_\alpha(K)$.\label{lengthexto}
\end{lem}
\begin{proof}
It suffices to prove that for every prime ideal $\PP$ of coheight $\alpha$, we have $\ell_\PP(M)=\ell_\PP(N)+\ell_\PP(K)$. In view of Lemma \ref{lengthext}, it is enough to obtain that the $A_\PP$-module $K\otimes A_\PP$ has finite length. Indeed, $K$ can be written as a composite extension of modules $A/\QQ_i$, where $\QQ_i$ are prime ideals of $A$. Then all $\QQ_i$ have coheight $\le\alpha$. Therefore either $\QQ_i=\PP$ or $\QQ_i$ is not contained in $\PP$. In the latter case, we have $A/\QQ_i\otimes A_\PP=0$, while $A/\PP\otimes A_\PP$ is the residual field of $A_\PP$ and therefore has length one. By flatness of $A_\PP$, we can thus write $K\otimes A_\PP$ as a composite extension of modules of length $\le 1$, so that $K\otimes A_\PP$ has finite length. 
\end{proof}

Let $M$ be a noetherian $A$-module. Define the ordinal-valued length function of $M$ as
$$\ell(M)=\sum_\alpha\omega^\alpha\cdot\ell_\alpha(M),$$
when the sum ranges over the ordinals $\alpha$ in \textit{reverse order}. 

Note that if $M$ has finite Krull dimension (as most usual noetherian modules) then the exponents in the above ``polynomial" are finite, i.e. $\ell(M)<\omega^\omega$.

The following proposition gives a characterization of the ordinal-valued length function $\ell$.

\begin{prop}
Let $A$ be a ring and $M$ a noetherian $A$-module. Then
$$\ell(M)=\sup\{\ell(N)+1|N\textnormal{ proper quotient of }M\}.$$\label{caractl}
\end{prop}

In other words, $\ell$ coincides with the ``Krull ordinal" of the Noetherian ordered set of submodules of $M$, introduced in \cite{Gull}, as well as (with a slight variant) in \cite{Bass}. This inductive definition makes sense more generally for any noetherian module over any ring (commutative or not); it can be viewed as a quantitative gauge of noetherianity.

\begin{proof}
Suppose that we have an exact sequence
$$0\to K\to M\to N\to 0,$$
with $K\neq 0$. Let $\alpha$ be the Krull dimension of $K$, and pick $\beta\ge\alpha$. Then, by Lemma \ref{lengthexto}, $\ell_\beta(M)=\ell_\beta(N)+\ell_\beta(K)$. In particular, if $\beta>\alpha$, then $\ell_\beta(M)=\ell_\beta(N)$, and $\ell_\alpha(M)>\ell_\alpha(N)$. Therefore $\ell(M)>\ell(N)$.

Now let us prove the other inequality, namely
$$\ell(M)\le\sup\{\ell(N)+1|N\textnormal{ proper quotient of }M\}.$$

\begin{itemize}
\item Zeroth case: $M=0$. Then we just get $0=\sup\emptyset$.

\item First case: $\ell(M)$ is a successor ordinal. This occurs if and only if $\ell_\PP(M)>0$ for some \textit{maximal} ideal $\PP$, in which case $\PP\in\Ass_A(M)$ by Lemma \ref{lengthass}, i.e. $M$ has a submodule $K$ isomorphic to $A/\PP$. It is then straightforward that 
$\ell(M)=\ell(M/K)+1$.

\item Second case: the least $\alpha$ such that $\ell_\alpha(M)\neq 0$ is a limit ordinal. Pick $\PP\in\Ass_A(M)$ with $\text{coht}(\PP)=\alpha$. Find an exact sequence $0\to K\to M\to N\to 0$ with $K\simeq A/\PP$. For every $\beta<\alpha$, there exists a prime ideal $\PP_\beta$ containing $\PP$ with $\text{coht}(\PP_\beta)=\beta$. Find a submodule $V_\beta$ of $K$ such that $K/V_\beta$ is isomorphic to $A/\PP_\beta$. From the exact sequences $0\to K\to M\to N\to 0$, $0\to K/V_\beta\to M/V_\beta\to N\to 0$ and Lemma \ref{lengthexto}, we get:
$$\ell_\gamma(M/V_\beta)=\ell_\gamma(M)\text{ if }\gamma>\alpha;$$
$$\ell_\alpha(M/V_\beta)=\ell_\alpha(M)-1;$$
$$\ell_\beta(M/V_\beta)\ge 1.$$
Now write $\ell(M)=P+\omega^\alpha$, where
$$P=\sum_{\gamma>\alpha}\omega^\gamma\cdot\ell_\gamma(M)+\omega^\alpha\cdot(\ell_\alpha(M)-1).$$
Then we get
$$\ell(M/V_\beta)\ge P+\omega^\beta,$$and thus
$$\sup_{\beta<\alpha}\ell(M/V_\beta)\ge P+\sup_{\beta<\alpha}\omega^\beta=P+\omega^\alpha=\ell(M).$$
\item Third case: the least $\alpha$ such that $\ell_\alpha\neq 0$ is a successor ordinal $\alpha=\beta+1$. Pick $\PP\in\Ass_A(M)$ with $\text{coht}(\PP)=\alpha$ and choose a prime ideal $\QQ$ of coheight $\beta$ containing $\PP$. Find an exact sequence $0\to K\to M\to N\to 0$ with $K\simeq A/\PP$. For every $n$, there exists a submodule $V_n$ of $K$ such that $K/V_n$ is isomorphic to $A/(\QQ^n+\PP)$. By Lemma
\ref{lengthinfinity}, $(\ell_\beta(K/V_n))$ is unbounded when $n\to\infty$. From the exact sequences $0\to K\to M\to N\to 0$, $0\to K/V_n\to M/V_n\to N\to 0$ and Lemma \ref{lengthexto}, we get:
$$\ell_\gamma(M/V_n)=\ell_\gamma(M)\text{ if }\gamma>\alpha;$$
$$\ell_\alpha(M/V_n)=\ell_\alpha(M)-1;$$
$$\ell_\beta(M/V_n)\to\infty\text{ when }n\to\infty,$$
and therefore $\sup_n\ell(M/V_n)\ge\ell(M)$.
\end{itemize}
\end{proof}


\section{Reduced length}\label{reducedo}

Define, for every ordinal $\alpha$, the ordinal $\alpha'$ as $\alpha'=\alpha+1$ if $\alpha<\omega$ and $\alpha'=\alpha$ otherwise.
If $M$ is a noetherian $A$-module, define its reduced length as follows

$$\ell'(M)=\sum_\alpha\omega^\alpha\cdot\ell_{\alpha'}(M),$$
where as usual the sum ranges over ordinal in reverse order. Observe that the reduced length is characterized by the length, as a consequence of the formula
$$\ell(M)=\omega\cdot\ell'(M)+\ell_0(M).$$

\begin{prop}
If $M$ is any noetherian $A$-module, then $\ell'(M)=\sup_N(\ell'(N)+1)$, where $N$ ranges over all quotients of $M$ with non-artinian kernel. Moreover, $\ell'(N)=\ell'(M)$ if $N$ is a quotient of $M$ with artinian kernel.\label{sigmaprime}
\end{prop}
\begin{proof}
The proof is similar to that of Proposition \ref{caractl}, so let us just sketch it, stressing on the differences.

First we have to prove that $\ell(N)<\ell(M)$ for every quotient $N$ of $M$ with non-artinian kernel $K$. The proof is the same, just noticing that then the Krull dimension of $K$ is at least one.

It remains to prove the reverse inequality
$$\ell'(M)\ge\sup\{\ell'(N)+1|N\textnormal{ quotient of }M\textnormal{ with non-artinian kernel}\}.$$

\begin{itemize}
\item Zeroeth case: $M$ is artinian. Then we just get $0=\sup\emptyset$.

\item First case: $\ell_1(M)\neq 0$. Then $M$ has a submodule $K$ isomorphic to $A/\PP$ for some prime ideal $\PP$ of coheight one. Then $\ell'(M)=\ell'(M/K)+1$.

\item Second case (respectively third case): the least $\alpha\ge 1$ such that $\ell_\alpha(M)\neq 0$ is a limit ordinal (resp. is a successor ordinal $\ge 2$). Go on exactly like in the case of non-reduced length.
\end{itemize}

Finally, if $N$ is a quotient of $M$ with artinian kernel, it follows from Lemma \ref{lengthexto} that $\ell_i(N)=\ell_i(M)$ for all $i\ge 1$, and therefore $\ell'(N)=\ell'(M)$.
\end{proof}


\section{Cantor-Bendixson rank}\label{sec:cb}

Let $M$ be a module over a ring $A$.
Let $\Sub_A(M)$ be the set of
submodules of $M$. This is a closed subset of $2^M$, endowed with the product topology which makes it a Hausdorff compact space; if $M$ has countable cardinality it is moreover metrizable. Let $\Quo_A(M)$ the set of quotients of $M$. It can be defined as a topological space that coincides with $\Sub_A(M)$, in which we view its elements as quotients of $M$ through the correspondence $K\leftrightarrow M/K$. In particular, $M\in\Quo_A(M)$ corresponds to $\{0\}\in\Sub_A(M)$. 

If $N$ is a quotient of $M$, then there is a natural embedding of $\Quo_A(N)$ into $\Quo_A(M)$. It is continuous so has closed image; moreover its image is open if and only if the $\text{Ker}(M\to N)$ is finitely generated (the easy argument is given in a similar context in \cite[Lemma~1.3]{CGP}); this is fulfilled if $M$ is Noetherian. 

In any topological space $X$ define by transfinite induction $I_0(X)$ as the set of isolated points of $X$, and $I_\alpha(X)=I_0(X-\bigcup_{\beta<\alpha}I_\beta(X))$. For $x\in X$, set
$\rk(x,X)=\infty$ if $x\notin\bigcup I_\alpha(X)$ and $\rk(x,X)=\alpha$ if $x\in I_\alpha$.
As all $I_\alpha(X)$ are pairwise disjoint, this is well-defined. Agree that $\infty>\alpha$ for every ordinal $\alpha$.

Now for every $A$-module $M$, set $\rk(M)=\rk(M,\Quo_A(M))$. Observe that if $N$ is a quotient of $M$, then $\rk(N)=\rk(N,\Quo_A(M))$, as the natural embedding $\Quo_A(N)\to\Quo_A(M)$ is open (because the kernel of $M\to N$ is finitely generated by noetherianity). In particular, if $N$ is generated by $k$ elements, then it can be viewed as a quotient of $A^k$, the free module of rank $k$, and $\rk(N)$ coincides with the Cantor-Bendixson rank of $N$ inside the space $\Quo_A(A^k)$ of finitely generated $A$-modules over $k$ marked generators.

\begin{lem}
Every noetherian $A$-module $M$ satisfying $\rk(M)=0$ (i.e. $M$ is isolated) has finite length. More precisely, every noetherian $A$-module of infinite length contains a decreasing sequence of non-zero submodules $(N_n)$ with trivial intersection.\label{rkzero}
\end{lem}

\begin{rem}
The converse is not true in general: for instance if $A$ is an infinite field and $M$ is a 2-dimensional vector space. However every module of finite cardinality is obviously isolated.
\end{rem}

\begin{proof}[Proof of Lemma \ref{rkzero}]
Otherwise, $M$ has a non-maximal associated ideal $\PP$. As it is clear that being isolated is inherited by submodules, we can suppose that $M=A/\PP$. We can even suppose that $\PP=\{0\}$, so that $A=M$ is a noetherian domain which is not a field. Then if $\mathcal{M}$ is a maximal ideal in $A$, then $\{0\}\neq\mathcal{M}^n\to \{0\}$ in $\Sub_A(A)$ and we get a contradiction.
\end{proof}

\begin{thm}
Let $M$ be a noetherian $A$-module. Suppose that every artinian subquotient of $M$ has finite cardinality. Then $\rk(M)=\ell'(M)$.\label{cb}
\end{thm}
\begin{proof}
Suppose that the statement is proved for every proper quotient of $M$.

Suppose that $\rk(M)>\ell'(M)$. Then there exists a sequence of proper quotients $M/W_n$, converging to $M$, such that $\rk(M/W_n)\ge\ell'(M)$. By induction, $\ell'(M/W_n)\ge\ell'(M)$. By Proposition \ref{sigmaprime}, this forces $W_n$ to be artinian, i.e. contained in the maximal artinian submodule $S$ of $M$. By assumption, $S$ is finite. As $W_n\to\{0\}$ in the space of submodules of $M$, this forces that eventually $W_n=0$, a contradiction.

Suppose that $\rk(M)<\ell'(M)$. Then, in view of Proposition \ref{sigmaprime}, there exists a quotient $M/W$, with $W$ non-artinian, such that $\ell'(M/W)\ge\rk(M)$. By induction we get $\rk(M/W)\ge\rk(M)$. As $W$ is non-artinian, it contains by Lemma \ref{rkzero} a properly decreasing sequence $(W_n)$ of non-zero submodules with trivial
intersection. Then $$\ell'(M/W_n)\ge\ell'(M/W)\ge\rk(M),$$ and $$\rk(M)\ge\sup(\rk(M/W_n)+1)$$ $$=\sup(\ell'(M/W_n)+1)\ge\rk(M)+1,$$
a contradiction.
\end{proof}

The following lemma is well-known, and implies that Theorem \ref{main} is a corollary of Theorem \ref{cb}.

\begin{lem}
Let $A$ be finitely generated ring. Then every finitely generated simple $A$-module has finite cardinality.
\end{lem}
\begin{proof}
Every such $A$-module $M$ can be viewed as a finitely generated $\mathbf{Z}$-algebra which is a field. Let $F$ be its prime subfield. Then by the Nullstellensatz, $M$ is a finite extension of $F$. If $F$ is a finite field we are done. If $F=\mathbf{Q}$, then $M=\mathbf{Q}[X]/P(X)$, where $P$ is a monic polynomial with coefficients in $\mathbf{Z}[1/k]$ for some integer $k>1$. Hence $M$ can be written as the increasing union of proper subrings
$\mathbf{Z}[1/n!k][X]/P(X)$, hence is not a finitely generated ring, a contradiction.
\end{proof}


\section{Comments on some other rings}\label{sec:other}

Theorem \ref{cb} fails when $A$ possesses an infinite simple $A$-module $M$ (i.e. $A$ has an infinite index maximal ideal). Indeed, we have $\ell'(M\times M)=0$, while $\rk(M\times M)=1$.

In all cases, we have:

\begin{prop}
$$\ell'(M)\le\rk(M)\le\ell(M).$$
\end{prop}
Here we set $(\alpha+1)-1=\alpha$, and $\alpha-1=\alpha$ if $\alpha$ is not a successor ordinal.

The left-hand inequality has already been settled in the proof of Theorem \ref{cb}, where we did not make use of the assumption on maximal ideals. The right-hand inequality is obtained by a straightforward induction. It is not optimal: for instance if $\ell(M)$ is a successor ordinal, it is easy to check that $\rk(M)\le\ell(M)-1$. I do not know to which extent this can be improved. However, the following example provides a quite unexpected behaviour.

Let $A$ be a local principal domain, with maximal ideal $I$ of infinite index. Denote the cyclic indecomposable $A$-modules $M_n=A/I^n$.

Let $T\simeq {M_1}^{n_1}\oplus\dots\oplus{M_k}^{n_k}$ be an $A$-module of finite length $\ell(T)=\sum in_i$, where $n_k\neq 0$. Set $\ell^*(T)=\ell(T)-k$. (Agree that $\ell^*(0)=0$.)

\begin{lem}
Let $M$ be a finitely generated $A$-module, and $T$ its torsion submodule. Then, for $n\ge 0$,
\begin{itemize}
\item if $M\simeq A^{2n}\oplus T$, then $\rk(M)=\omega\cdot n+\ell^*(T)$;
\item if $M\simeq A^{2n+1}\oplus T$, then $\rk(M)=\omega\cdot n+\ell(T)+1$.
\end{itemize}
\end{lem}

Observe that on the other hand, $\ell(A^k\oplus T)=\omega\cdot k+\ell(T)$, and $\ell'(A^k\oplus T)=\ell(T)$.

\begin{proof}
Let us argue by induction on $\ell(M)$.

Suppose that $M\simeq A^{2n}\oplus T$ with $\ell^*(T)\neq 0$. Then $T\simeq {M_1}^{n_1}\oplus\dots\oplus{M_k}^{n_k}$, where $n_k\neq 0$ and $\sum n_i\ge 2$. Inside the socle of $T$, one can find infinitely many cyclic submodules $D_j$, with $D_j\cap D_m=\{0\}$ for $n\neq m$. Thus 
$(M/D_j)$ is a sequence of distinct modules tending to $M$. Now by induction $\rk(M/D_j)=\omega\cdot n+\ell^*(T/D_j)=\omega\cdot n+\ell^*(T)-1$ (except maybe for one single value of $j$, if $n_k=1$),
so $\rk(M)\ge\omega\cdot n+\ell^*(T)$.

Suppose now that $M\simeq A^{2n}\oplus T$ (with in mind $\ell^*(T)=0$ although we do not need it). If $n\ge 1$, then $M$ is a limit for $d\to\infty$ of modules isomorphic to $A^{2n-1}\oplus M_d\oplus T$, which by induction have $\rk\ge\omega\cdot (n-1)+d$. So $\rk(M)\ge\omega\cdot n$. This also obviously holds if $n=0$.

Suppose that $M\simeq A^{2n+1}\oplus T$. Then $M$ is a limit of modules isomorphic to $A^{2n}\oplus M_d\oplus T$, for $d\to\infty$, which by induction have $\rk\ge\omega\cdot n+\ell^*(M_d\oplus T)=\omega\cdot n+\ell(T)$ for $d$ large enough. Thus $\rk(M)\ge\omega\cdot n+\ell(T)+1$.

So we have established, in all cases, $\rk(M)\ge r(M)$, where $r(M)$ is the right-hand term given by the proposition. Let us now prove the upper bound $\rk(M)\le r(M)$, again by induction.

Suppose that $M=A^{2n}\oplus T$, and let $M/N$ be a quotient of $M$. If $N$ is not contained in $T$, then the rank drops, so that by induction $\rk(M/N)<r(M)$. Otherwise, define $k$ as in the beginning of the proof. Then, provided that $N$ does not contain the socle of ${M_k}^{n_k}$, we have $\ell^*(M/N)<\ell(M)$, so that $\rk(M/N)<r(M)$. 
We claim that at the neighbourhood of $\{0\}$ in $\Sub_A(M)$, no $N$ contains the socle of $M$. Indeed, fix a nonzero element $m$ in this socle (a ``discriminator"). The set of submodules $N$ in which $m\notin N$ is an open neighbourhood of $\{0\}$ in $\Sub_A(M)$ (which corresponds to a neighbourhood of $\{M\}$ in $\Quo_A(M)$); no such $N$ contains the socle. Thus we obtain $\rk(M)\le r(M)$.

Suppose that $M=A^{2n+1}\oplus T$, and let $M/N$ be a proper quotient of $M$. Then it is straightforward that $r(M/N)<r(M)$, and therefore we get $\rk(M)\le r(M)$. This concludes the proof.
\end{proof}


\end{document}